\documentclass[12pt]{article}

\usepackage{amsrefs}
\usepackage{amsmath,amssymb,amsthm}
\usepackage{enumerate,pxfonts,tikz}
\usepackage[all,cmtip]{xy}
\usepackage[applemac]{inputenc}
\usepackage[colorlinks=true, linkcolor=black, citecolor=black]{hyperref}

\newcommand \al{\alpha}
\newcommand \Aut{\operatorname{Aut}}
\newcommand \bs{\backslash}
\newcommand \C{{\mathbb C}}

\newcommand \CC{\mathcal{C}}

\newcommand \CK{\mathcal{K}}

\newcommand \CQ{\mathcal{Q}}

\newcommand \CT{\mathcal{T}}
\newcommand \Cone{\operatorname{Cone}}

\newcommand \D{{\mathbb D}}
\newcommand \ds{\displaystyle}
\newcommand \End{\operatorname{End}}
\newcommand \eps{\varepsilon}

\newcommand \Ga{\Gamma}

\newcommand \ga{\gamma}

\newcommand \la{\lambda}
\newcommand \La{\Lambda}
\newcommand \lab{\operatorname{lab}}
\newcommand \mqed{\tag*\qedhere}
\newcommand \N{{\mathbb N}}

\newcommand \R{{\mathbb R}}
\newcommand \sm{\smallsetminus}

\newcommand \tr{\operatorname{tr}}

\newcommand \val{\operatorname{val}}

\newcommand \Z{{\mathbb Z}}

\renewcommand \a{\mathfrak a}
\renewcommand \b{\mathfrak b}

\renewcommand \({\left(}
\renewcommand \){\right)}

\newcommand{\e}
[1]{\emph{#1}\index{#1}}

\newcommand{\norm}
[1]{\left\|#1\right\|}

\renewcommand{\sp}
[1]{\left\langle #1\right\rangle}

\newcommand{\tto}
[1]{\stackrel{#1}{\longrightarrow}}

\newtheorem{theorem}{Theorem}[section]

\newtheorem{lemma}[theorem]{Lemma}

\newtheorem{proposition}[theorem]{Proposition}

\theoremstyle{definition}
\newtheorem{definition}[theorem]{Definition}
\newtheorem{example}[theorem]{Example}

\begin{document}

\pagestyle{myheadings} \markright{WEIGHTED PRIME GEODESIC THEOREMS}

\title{Weighted prime geodesic theorems}
\author{Anton Deitmar}
\date{}
\maketitle

{\bf Abstract:} Prime geodesic theorems for weighted infinite graphs and weighted building  quotients are given.
The growth rates are expressed in terms of the spectral data of suitable translation operators inspired by a paper of Bass.

$ $

{\bf MSC: 11N80}, 11F72, 11N05, 20E08, 20F65, 51E24, 53C22

$$ $$

\tableofcontents

\newpage
\section*{Introduction}

The first Prime Geodesic Theorem was given by Huber in \cite{Huber}.
It states that for a compact hyperbolic surface $X$ the number $N(T)$ of prime closed geodesics of lenght $\le T$ satisfies
$$
N(T)\sim \frac{e^{2T}}{2T},\quad\text{as}\quad T\to\infty.
$$
It has been sharpened by giving estimates on the error term and it has been extended to other manifolds \cite{Hejhal,Iwaniec,Katsunada,Koyama,ADhigherRank,Sound}.
Applications to class numbers are  in \cite{Sarnak}, and, extending this result, in \cite{class,classNC}.
It was extended to more general dynamical systems, culminating in Margulis's celebrated result on Anosov flows, stating that the number $N(T)$ of closed orbits of length $\le T$ satisfies $N(T)\sim\frac{e^hT}{hT}$, where $h$ is the entropy of the system \cite{Margulis}.

An extension to the graph case was formulated in \cite{Hashi,ADCont}.
In \cite{weightedIhara}, zeta functions of graphs with weights were introduced.
Here weights are representing resistance to a flow along the edges  or a individual distribution of the flow at the nodes.
In this paper we present a notion flexible enough to take into account these different requirements.

The paper \cite{PGTbuild} gives Prime Geodesic Theorems for  compact quotients of buildings, generalizing the graph case.
This generalization is natural, as in number theoretical situations, graphs and building quotients both turn up in the $p$-adic setting.
In the present paper the latter two ideas are combined in stating  Prime Geodesic Theorems  for weighted infinite graphs and weighted building quotients.
In the latter case, a full expansion of the numbers of closed geodesics of a given length  is presented.

The first section treats an extension of the Perron-Frobenius Theorem suitable for our purposes.
The next two sections deal with the graph case and the last three with building quotients.

\section{$E$-positive operators}\label{sec1}

\begin{definition}
Let $H$ be a Hilbert space and $E$ an orthonormal basis. 
Let $H^+_E$ denote the cone of all $\sum_{e\in E}c_ee\in H$ with $c_e\ge 0$ for every $e\in E$.
A bounded linear operator $T$ on $H$ is called \e{$E$-positive}, if $T(H^+_E)\subset H^+_E$.
This is equivalent to
$$
\sp{Te,f}\ge 0
$$
for all $e,f\in E$.
Note that the adjoint $T^*$ is $E$-positive if $T$ is.
Further, if $S$ and $T$ are $E$-positive, then so is $ST$.
\end{definition}

\begin{definition}
We say that a bounded  operator $T:H\to H$ is \e{$E$-reducible}, if there exists a proper subset 
$\emptyset\ne F\subsetneq E$ such that
$$
T(F)\subset\ell^2(F),
$$
where by $\ell^2(F)$ we mean the closure of the span of $F$.
This is in accordance with the isomorphism $H\cong\ell^2(E)$.
In this case, the closed subspace $\ell^2(F)$ is $T$-stable.

If $T$ is not $E$-reducible, we say that $T$ is \e{$E$-irreducible}.
\end{definition}

\begin{definition}
For a compact operator $T$ and an eigenvalue $\la$ the \e{geometric multiplicity} is the dimension of the eigenspace $\ker(T-\la)$.
In this situation, the sequence $\ker(T-\la)^n$, $n\in\N$ is eventually stationary.
The \e{algebraic multiplicity} is the limit
$$
\lim_{n\to\infty}\dim\ker(T-\la)^n\in\N.
$$
\end{definition}

\begin{theorem}
Let $T$ be a positive compact operator on $H$ with positive spectral radius $r=r(T)>0$.
Then the spectral radius $r$ is an eigenvalue of $T$.
Assume that $T$ is $E$-irreducible and trace class. Then the algebraic multiplicity of $r$ is 1.
The eigenvalues $\la$ with $|\la|=r$ distribute evenly over the unit circle times $r$. More precisely, if $\la_0=r,\la_1,\dots,\la_{n-1}$ are all eigenvalues of $T$ with $|\la_j|=r$, then we can order them in a way that $\la_j=e^{2\pi ij/n}$.
Finally, every $\la_j$ has algebraic multiplicity 1.
\end{theorem}

\begin{proof}
If $\dim H<\infty$, this is the Theorem of Frobenius, see \cite{Gant}.
So we assume $H$ to be infinite-dimensional.
We first show that if $T$ is $E$-irreducible, then so is the adjoint operator $T^*$.
For this assume $T$ $E$-irreducible and let $F\subset E$ be such that $T^*$ maps $\ell^2(F)$ to itself.
Then for every $f\in F$ and every $e\in E\sm F$ we have
$$
0=\sp{T^*f,e}=\sp{f,Te}.
$$
This implies that $T$ maps $\ell^2(E\sm F)$ to itself, hence $F$ is either empty or equals $E$, which means that $T^*$ is $E$-irreducible.

The fact that $r=r(A)$ is an eigenvalue is known as the \e{Krein-Rutman Theorem}, see for instance Theorem 7.10 in \cite{Abram}.
This theorem also says that there exists an eigenvector $v_0$ for the eigenvalue $\la_0=r$ which is $E$-positive in the sense that $\sp{v_0,e}\ge 0$ for every $e\in E$.
Let $F$ be the set of all $f\in E$ such that $\sp{v_0,f}=0$.
For $f\in F$ we write $T^*f=\sum_{e\in E}c_ee$. have
$$
0=r\sp{v_0,f}=\sp{Tv_0,f}=\sp{v_0,T^*f}=\sum_{e\in E} c_e\sp{v_0,e}.
$$
Since $T^*$ is positive, $c_e\ge 0$ and so we have $e\in F$ if $c_e\ne 0$.
This means that $\ell^2(F)$ is $T^*$-stable, hence, since $v_0\ne 0$, we get $F=\emptyset$.
So we conclude that $\sp{v_0,e}>0$ for every $e\in E$. 
This means that $v_0$ is a \e{totally $E$-positive} vector. 

We now assume that $T$ is $E$-irreducible, $E$-positive and trace class.
Then the Fredholm determinant $\det(1-uT)$ is an entire function.
For $u\in\C$ such that $\frac1u\notin \sigma(T)$ we set
$$
B(u)=\det(1-uT)(1-uT)^{-1}.
$$
For each  $\la\in\C$ let
$H(\la)$ be the largest $T$-stable subspace on which $T$ has spectrum $\{\la\}$.
If $\la\ne 0$, then $H(\la)$ is finite-dimensional and
$$
H(\la)=\bigcup_{n=1}^\infty \ker(T-\la)^n.
$$
Let $H^r=\bigoplus_{\la\ne r}H(\la)$.
Then $H(r)$ and $H^r$ are closed, $T$-stable subspaces and
$$
H=H(r)\oplus H^r.
$$
Then $\det(1-uT|_{H(r)})(1-uT|_{H(r)})^{-1}=(1-uT|_{H^r)})^\#$,
where the $\#$ indicates the adjugate matrix.
This implies that $B(u)$ extends continuously to $u=1/r$.
Applying $B(1/r)$ to a vector in $H_{<r}$ we see that $B(1/r)\ne 0$.
For $0<u<1/r$ and $e,f\in E$ we have
\begin{align*}
\sp{B(u)e,f}&=\det(1-uT)\sp{(1-uT)^{-1}e,f}\\
&=\underbrace{\det(1-uT)}_{\in\R}\underbrace{\sum_{n=0}^\infty u^u\sp{T^ne,f}}_{\ge 0}.
\end{align*}
Letting $u$ tend to $1/r$ we conclude that there is a sign $\sigma\in\{\pm1\}$ such that $\sigma B(1/r)$ is $E$-positive.
Since $v_0$ is totally positive, we have
$$
0\ne B(1/r)v_0=\lim_{u\nearrow 1/r}\det(1-uT)(1-uT)^{-1}v_0=\lim_{u\nearrow 1/r}\det(1-uT)\frac1{1-ur}v_0.
$$
This implies that $1/r$ is a simple zero of $\det(1-uT)$, which is to say that the algebraic multiplicity is 1.
Therefore the condition (G) of Definition 4.7 in \cite{Schae} is satisfied, and hence the remaining points of the theorem follow from Theorem 5.2 in \cite{Schae}.
\end{proof}

We finally consider the situation without the condition of irreducibility.

\begin{proposition}\label{prop1.5}
Let $T$ be a compact operator with positive spectral radius $r(T)>0$.
Then for a given orthonormal basis $E$ there exists a disjoint decomposition
$E=E_1\cup\dots\cup E_n$ with the following properties:
\begin{enumerate}[\rm (a)]
\item The space $\ell^2(E_1\cup\dots\cup E_k)$ is $T$-stable for each $k$.
\item
For each $k$ let  $P_k$ denote the orthogonal projection onto $\ell^2(E_k)$ then the operator
$$
T_k=P_kT:\ell^2(E_k)\to\ell^2(E_k)
$$
either has a spectral radius $r(T_j)<r(T)$ or is $E$-irreducible.
\end{enumerate}
If $T$ is $E$-positive, then $T_k$ is $E_k$-positive for each $k$.
If $T$ is trace class, then each $T_k$ is and for the Fredholm determinant one has
$$
\det(1-uT)=\prod_{k=1}^n\det(1-uT_k),\quad u\in\C.
$$
\end{proposition}

\begin{proof}
If $T$ is $E$-irreducible, we are done.
Otherwise there is a $\emptyset\ne F\subsetneq E$ such that $\ell^2(F)$ is $T$-stable.
We consider $T_F=T|_{\ell^2(F)}$ and $T_{E\sm F}=PT:\ell^2(E\sm F)\to\ell^2(E\sm F)$, where $P$ is the orthogonal projection to $\ell^2(E\sm F)$.
We have $\det(1-uT)=\det(1-uT_F)\det(1-uT_{E\sm F})$. If either of the operators $T_F$ or $T_{E\sm F}$ has spectral radius $<r(T)$ or ist irreducible, we leave this factor in peace and continue with the other, which we then decompose further.
This process will stop, as there are only finitely many spectral values $\la$ with $|\la|=r(T)$ and these are eigenvalues of finite multiplicity. By the fact that the Fredholm determinant distributes it follows that the spactral values distribute and by the finiteness, this process will stop, yielding the proposition. 
\end{proof}

\section{The Ihara Zeta function}
\begin{definition}
Let $X$ denote an oriented graph. 
This means that $X$ consists of the following data: a set $N(X)$ of \e{nodes} and a set $E(X)\subset X\times X$ of \e{oriented edges}.
We call $x$ the \e{source} of the oriented edge $e=(x,y)$ and $y$ the \e{target} and we write this as $x=s(e)$, $y=t(e)$.
The nodes $x,y$ are called the \e{endpoints} of the edge $(x,y)$.
\end{definition}

\begin{definition}
The \e{valency}, $\val(x)$ of a node $x$ is the number of edges having $x$ for one of their endpoints.
Throughout, we will assume that $X$ has \e{bounded valency}, i.e., that there exist $M>0$ such that
$$
\val(x)\le M
$$
holds for all nodes $x$.
\end{definition}

We introduce the notion of a weight.
Usually, a weight $w(e)>0$ is put on an edge representing a length or a resistance or the reciprocal of that.
In order to be more flexible and also consider paths with partial backtracking, it is more convenient for us to consider a \e{transition weight}, which may be interpreted as the likelihood of a particle or a current, of choosing a certain edge.

\begin{definition}
A \e{transition weight} on $X$, henceforth simply called a weight, is a map 
$$
w:E(X)\times E(X)\to [0,\infty)
$$
such that
$$
w(e,f)\ne 0\quad\Rightarrow\quad t(e)=s(f).
$$
and
$$
\sum_{e,f\in OE(X)}w(e,f)<\infty,
$$
\end{definition}

\begin{example}
A natural example is given in the case of $X$ being a quotient $X=\Ga\bs Y$, where $Y$ is a tree of bounded valency and $\Ga$ is a \e{tree lattice} \cite{treelat}.
In this case the weight
$$
w(e,f)=\frac1{|\Ga_f|}
$$
is a natural choice which fits the approach of Bass \cite{Bass} to the Ihara zeta function in case of ramified quotients, see also \cite{treelattice}.
\end{example}

\begin{definition}
An \e{oriented path} of \e{length} $n$ in $X$ is an $n$-tuple $p=(e_1,e_2,\dots,e_n)$ of oriented edges such that $t(e_j)=s(e_{j+1})$ holds for all $j=1,\dots,n-1$.
We write the length  as $\ell(p)=n$.
The path $p$ is called a \e{closed path}, if $t(e_n)=s(e_1)$.
For a closed path $p=(e_1,\dots,e_n)$ we define its \e{weight} as
$$
w(p)=w(e_1,e_2)w(e_2,e_3)\cdots w(e_{n-1},e_n)w(e_n,e_1).
$$
\end{definition}

\begin{definition}
(Shifting the starting point)
On the set $CP(X)$ of all closed paths we instal an equivalence relation $\sim$ generated by
$$
(e_1,e_2,\dots,e_n)\ \sim\ (e_2,e_3,\dots,e_n,e_1).
$$
An equivalence class $c=[p]$ of closed paths is called a \e{cycle}.
The length and weight functions factor through the quotient of this equivalence, so $\ell(c)$ and $w(c)$ are well-defined for a cycle $c$.
\end{definition}

\begin{definition}
For a cycle $c$ and a natural number $k$ we define $c^k$ to be the cycle one gets by iterating the cycle $c$ for $k$-times.
One has
$$
\ell(c^k)=k\ell(c),\quad\text{and}\quad w(c^k)=w(c)^k.
$$
A cycle $c$ is called \e{primitive}, if $c$ is not a power $c_1^k$ of some shorter cycle $c_1$.
For every cycle $c$ there is a uniquely determined primitive $c_0$ and a uniquely determined number $\mu(c)\in\N$ such that
$$
c=c_0^{\mu(c)}.
$$
The cycle $c_0$ is called the \e{underlying primitive} and $\mu(c)$ is called the \e{multiplicity} of the cycle $c$.
\end{definition}

\begin{definition}
The \e{Ihara zeta function} of the weighted oriented graph $(X,w)$ is defined by the  product
$$
Z(u)=\prod_{c}\(1-w(c)u^{\ell(c)}\)^{-1},
$$
where the product runs over all primitive cycles in $X$.
\end{definition}

\begin{definition}
Let $H=\ell^2(E)$ be the Hilbert space of all $\ell^2$-functions on $E(X)$ the elements of which we write as formal series $\sum_{e\in OE(X)}c_e e$ with $c_e\in\C$ satisfying $\sum_{e\in OE(X)}|c_e|^2<\infty$.
\end{definition}

\begin{definition}
Inspired by \cite{Bass}, we consider the \e{Bass operator}  $T: H\to H$ given by
$$
T(e)=\sum_{f\in OE(X)}w(e,f)f.
$$
If for two edges we have $t(e)=s(f)$, we write  $e\to f$, so we have $T(e)=\sum_{e\to f}w(e,f)f$.
\end{definition}

The following theorem is a straightforward generalization of Theorem 1.6 of \cite{weightedIhara}.

\begin{theorem}
The operator $T$ is of trace class.
The  product $Z(u)$ converges for $|u|$ sufficiently small.
The function $Z(u)$ extends to a meromorphic function, more precisely, $Z(u)^{-1}$ is entire and satisfies 
$$
Z(u)^{-1}=\det(1-uT),
$$
where $\det$ is the Fredholm-determinant.
\end{theorem}

\begin{proof}
The proof is essentially the same as the proof of Theorem 1.6 of \cite{weightedIhara}.
We repeat it here for the convenience of the reader.
We show that the operator $T$ is of trace class, and for every $n\in\N$ we have
$$
\tr T^n= \sum_{l(c)=n} l(c_0)\,w(c),
$$
where the sum runs over all  cycles $c$ of length $n$ and $c_0$ is the underlying primitive cycle to $c$.

For this we consider the natural orthonormal basis of $\ell^2(E)$  given by $E$.
Using this orthonormal basis, one easily sees that $T^n$ has the claimed trace, once we know that $T$ is of trace class.
For this we estimate
\begin{align*}
\sum_{e\in E}\norm{Te}&=\sum_e\(\sp{Te,Te}\)^{\frac12}
=\sum_e\(\sum_{f\in E}\sp{Te,f}\sp{f,Te}\)^{\frac12}\\
&\le \sum_e\sum_f|\sp{Te,f}|^2=\sum_{e,f}w(e,f)^2<\infty.
\end{align*}
This implies that $T$ is of trace class.
For small values of $u$ we have
\begin{align*}
\det(1-uT) &=\exp\left( -\sum_{n=1}^\infty\frac{u^n}n\tr T^n\right)
=\exp\left( -\sum_{n=1}^\infty\frac{u^n}n\sum_{l(c)=n}l(c_0)\,w(c)\right)\\
&=\exp\left( -\sum_{c_0}\sum_{m=1}^\infty \frac{u^{ml(c_0)}}m \,w(c_0)^m\right)
= \prod_{c_0}\left( 1-w(c_0)u^{l(c_0)}\right)= Z(u)^{-1}.
\end{align*} 
This a fortiori also proves the convergence of the product.
\end{proof}

\section{The prime geodesic theorem for a weighted graph}

\begin{theorem}\label{thm3.1}
Let $(X,w)$ be a weighted oriented graph.
For $m\in\N$ let
$$
N_m=\sum_{c:l(c)=m}w(c)\,l(c_0).
$$
Then there are $r>0$ and  natural numbers $n_1,\dots, n_s$ such that, as $m\to\infty$,
$$
N_m=r^m\sum_{k=1}^sn_k{\bf 1}_{n_k\N}(m)+O((r-\eps)^m)
$$
for some $\eps>0$.
\end{theorem}

\begin{proof}
Let $Z(u)$ the Ihara zeta function of $X$.
By the definition of $Z$ we get
$$
u\frac{Z'}Z(u)=\sum_{m=1}^\infty N_mu^m.
$$
The formula $Z(u)=\det(1-uT)^{-1}$  on the other hand yields
$$
N_m=r^m\sum_{k=1}^s\sum_{j=0}^{n_k-1}e^{\frac{2\pi ijm}{n_k}}+\sum_{|\la|<r} m(\la)\la^m,
$$
where $r=r(T)>0$ is the spectral radius of $T$ and the last sum runs over all eigenvalues $\la$ of $T$ with $|\la|<r$. The number $m(\la)$ is the algebraic multiplicity of $\la$. 
Finally, the numbers $s$ is the number of components as in Proposition \ref{prop1.5} with $r(T_j)=r(T)$ and $n_j$ is the number of eigenvalues $\la$ of that component, which satisfy $|\la|=r(T)$.
Since the sum $\sum_{j=0}^{n_k-1}e^{\frac{2\pi ijm}{n_k}}$ is zero unless $m$ is a multiple of $n_k$, the theorem follows.
\end{proof}

Let 
\begin{align*}
\vartheta(n)&=\sum_{c_0:l(c_0)\le n}w(c_0)l(c_0),\\
\psi(n)&=\sum_{c:l(c)\le n}w(c)l(c_0),\\
\pi(n)&=\sum_{c_0:l(c_0)\le n}w(c_0).
\end{align*}

\begin{proposition}
Assume that $r=r(T)>1$.
Let $n_1,n_2,\dots,n_s$ the numbers of Theorem \ref{thm3.1}, let $K$ be their least common multiple and let $C=\sum_{k=1}^s\frac{n_kr^{n_k}}{r^{n_k}-1}$. 
Then, as $n\to\infty$ one has
$$
\vartheta(nK)\sim\psi(nK)\sim r^{nK}C
$$
and
$$
\pi(nK)\sim\frac{r^{nK}}{nK}C.
$$
\end{proposition}

\begin{proof}
Let 
\begin{align*}
l_1&=\liminf_n\frac{\vartheta(nK)}{r^{nK}},&L_1&=\limsup_n\frac{\vartheta(nK)}{r^{nK}},\\
l_2&=\liminf_n\frac{\psi(nK)}{r^{nK}},&L_1&=\limsup_n\frac{\psi(nK)}{r^{nK}},\\
l_3&=\liminf_n\frac{nK\pi(nK)}{r^{nK}},&L_1&=\limsup_n\frac{nK\pi(nK)}{r^{nK}},\\
\end{align*}
We compute
\begin{align*}
\vartheta(n)&\le\psi(n)=\sum_{c_0:l(c_0)\le n}\sum_{j=1}^{\left[\frac n{l(c_0)}\right]}w(c_0)^jl(c_0)\\
&\le n\sum_{c_0:l(c_0)\le n}w(c_0)=n\pi(n)
\end{align*}
and this implies $l_1\le l_2\le l_3$ as well as $L_1\le L_2\le L_3$.
Next we let $0<\al<1$ and we get
\begin{align*}
\vartheta(n)&\ge \sum_{\al n< l(c_0)\le n} w(c_0)l(c_0)\ge \al n\sum_{\al n< l(c_0)\le n}w(c_0)\\
&=\al n(\pi(n)-\pi(\al n)).
\end{align*}
So that
\begin{align*}
\frac{\vartheta(n)}{r^n}\ge \al\frac{n\pi(n)}{r^n}-\frac{\al n\pi(\al n)}{r^{\al n}} r^{(\al-1)n}.
\end{align*}
Since $r^{(\al-1)n}$ tends to zero, this implies and this finally implies $\al l_3\le l_1$ and $\al L_3\le L_1$.
As $\al$ was arbitrary it follows $l_3\le l_1$ and $L_3\le L_1$, so $l_1=l_2=l_3$ and $L_1=L_2=L_3$.
The claim follows if we finally show that $l_2=L_2=C$.
We have
\begin{align*}
\psi(nK)&=\sum_{m=1}^{nK}N_m=\sum_{m=1}^{nK} r^m\sum_{k=1}^sn_k{\bf 1}_{n_k\N}(m)+O((r-\eps)^m)\\
&=\sum_{m=1}^{nK}r^m\sum_{k:n_k|m}n_k+O((r-\eps)^m)\\
&=\sum_{k=1}^s n_k \sum_{\substack{m\le nK\\ n_k|m}}r^m+O((r-\eps)^m)\\
&=\sum_{k=1}^sn_k\sum_{1\le j\le \frac {nK}{n_k}}r^{jn_k}+O((r-\eps)^{jn_k})\\
&=\sum_{k=1}^sn_kr^{n_k}\frac{r^{nK}-1}{r^{n_k}-1}+O((r-\eps)^{nK})
\end{align*}
and so
\begin{align*}
\frac{\psi(nK)}{r^{nK}}&=\sum_{k=1}^sn_kr^{n_k}\frac{1-r^{-nK}}{r^{n_k}-1}+O((r-\eps)^{nK})\\
&=\sum_{k=1}^sn_k\frac{r^{n_k}}{r^{n_k}-1}+o(1)
\mqed
\end{align*}
\end{proof}

\begin{proposition}
\begin{enumerate}[\rm (a)]
\item For the spectral radius $r=r(T)$ one has
$$
r\ =\ \frac1{\sup\{ u>0:\sum_{c}w(c)u^{l(c)}<\infty\}}.
$$
\item
If there exist two different primitive cycles $c_0, d_0$ with a common oriented edge $e$ and $w(c_0)=w(d_0)\ge1$, then $r>1$.
\end{enumerate}
\end{proposition}

\begin{proof}
(a) We have 
$$
u\frac{Z'}Z(u)=\sum_{c}w(c)l(c_0)u^{l(c)}.
$$
Hence it follows that $\frac1r$ equals the supremum of all $u>0$ for which $\sum_{c}w(c)l(c_0)u^{l(c)}<\infty$.
Let $l=\sup\{ u>0:\sum_{c}w(c)u^{l(c)}<\infty\}$.
We show that $l=\frac1r$.
So suppose that $\sum_{c}w(c)u^{l(c)}<\infty$.
Then, as power series may be differentiated element-wise and on the other hand they converge locally uniformly and may be integrated, it follows that
$$
\sup\{ u>0:\sum_{c}w(c)u^{l(c)}<\infty\}
=\sup\{ u>0:\sum_{c}w(c)l(c)u^{l(c)}<\infty\}.
$$
As $1\le l(c_0)\le l(c)$, the claim follows.

(b) 
The cycles $c_0d_0$ and $d_0c_0$ are primitive, distinct, and of the same length.
Replacing $c_0$, $d_0$ with these, we assume that $l(c_0)=l(d_0)=l$.
Next we fix representing closed paths of $c_0$ and $d_0$ which start with the edge $e$.
For $n\in\N$ we have
$$
\sp{T^{nl}e,e}\ge\sum_{p:l(p)=nl}1,
$$
where the sum runs over all closed paths starting with $e$ and being of the form $c_0^{n_1}d_0^{m_1}\cdots c_0^{n_s}d_0^{m_s}$ for some $m_jn_j\in\N_0$.
This implies that
$$
\sp{T^{nl}e,e}\ge\sp{S^{3n}f,f},
$$
where $S$ is the Bass operator of the oriented graph with constant weight $w=1$
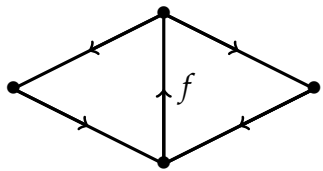
\begin{center}
\begin{tikzpicture}
\draw(0,0)node{$\bullet$};
\draw(0,2)node{$\bullet$};
\draw(-2,1)node{$\bullet$};
\draw(2,1)node{$\bullet$};
\draw[->,very thick](0,0)--(0,1);
\draw[very thick](0,0)--(0,2);
\draw[very thick](0,2)--(2,1);
\draw[->,very thick](0,2)--(1,1.5);
\draw[very thick](2,1)--(0,0);
\draw[->,very thick](2,1)--(1,.5);
\draw[very thick](0,0)--(2,1);
\draw[->,very thick](0,2)--(-1,1.5);
\draw[very thick](0,2)--(-2,1);
\draw[->,very thick](-2,1)--(-1,.5);
\draw[very thick](0,0)--(-2,1);
\draw(.3,1)node{$f$};
\end{tikzpicture}
\end{center}
One sees that $S^{3n}f=2^nf$
Therefore $\sp{T^{nl}e,e}\ge 2^n$ and so
the operator norm satisfies $\norm T^{nl}\ge 2^n$.
The spectral radius therefore satisfies
\begin{align*}
r=\lim_n\norm{T^{nl}}^\frac1{nl}\ge 2^{\frac1l}>1.
\mqed
\end{align*}\end{proof}

\section{Affine buildings}
For background on this section, the reader may consult 
\cite{bldglat}.
Let $X$ be a locally finite affine building.
By this we understand a polysimplicial complex which is the union of a given family of affine Coxeter complexes, called apartments, such that any two chambers (=cells of maximal dimension, which is fixed) are contained in a common apartment and for any two apartments $\a,\b$ containing chambers $C,D$ there is a unique isomorphism $\a\to\b$ fixing $C$ and $D$ point-wise.
A chamber is called \e{thin} if at every wall it has a unique neighbor chamber, it is called \e{thick}, if at each wall it neighbors at least two other chambers. 
The building is called thin or thick if all its chambers are.
For the ease of presentation, we will always assume that the building $X$ is  simplicial instead of polysimplicial.

Note that our definition includes buildings which are not Bruhat-Tits.
In higher dimensions, buildings tend to be of Bruhat-Tits type \cite{BruhatTits}.
For buildings of dimension at most two the situation is drastically different.
Indeed, 
Ballmann and Brin proved that every 2-dimensional simplicial complex in which the links of vertices are isomorphic to the flag complex of a finite projective plane has the structure of a building \cite{BallmannBrin}.

When speaking of ``points'' in $X$, we identify the complex $X$ with its geometric realization.
Note that the latter carries a topology as a CW-complex.
In this topology, a set is compact if and only if it is closed and contained in a finite union of chambers.
Note that an affine building is always contractible, see Section 14.4 of \cite{Garrett}.

\begin{definition}
Generally, there are different families of apartments which make $X$ a building, but there is a unique maximal family (Theorem 4.54 of \cite{Bldgs}).
In this paper, we will always choose the maximal family.
Let $\Aut(X)$ be the automorphism group of the building $X$, that is, the set of all automorphisms $g:X\to X$ of the complex $X$ which map apartments to apartments.
In the geometric realisation these are cellular maps which are affine on each cell. 
\end{definition}

\begin{definition}\label{def2.4.2}
A choice of \e{types} is a labelling that attaches to each vertex $v$ a label $\lab(v)$, or type in $\{0,1,\dots,d\}$ such that for each chamber $C$ the set $V(C)$ of vertices of $C$ is mapped bijectively to $\{0,1,\dots,d\}$.

Restricting the labelling gives a bijection between the set of all choices of types and the set of all bijections $V(C_0)\tto\cong\{0,1,\dots,d\}$, where $C_0$ is any given chamber.
Therefore the number of different choices of types is $(d+1)!$. 
We fix a choice of types such that each vertex of type zero is a special vertex.
This means that the set of reflection hyperplanes containing  it, meets every parallelity class of reflection hyperplanes of the ambient apartment, see Definition 1.2.3 of \cite{bldglat}.
\end{definition}

\begin{definition}
Pick a chamber $C$ and an apartment $\a$ containing $C$.
Let $v_0,v_1,\dots,v_d$ be the vertices of $C$ with $\lab(v_j)=j$.
Pick $v_0$ as origin to give the affine space $\a$ the structure of a vector space.
Let $\Cone(C,\a)$ denote the open cone in $\a$ spanned by the interior $\mathring C$ of the chamber $C$, i.e.
$$
\Cone(C,\a)=(0,\infty)\cdot\mathring C.
$$
Further let
$$
\Cone(C)=\bigcup_{\a\supset C}\Cone(C,\a).
$$
For two chambers $C,D$ we finally write 
$$
C\leadsto D
$$
if and only if
$$
C\ne D\quad\text{and}\quad \Cone(D)\subset\Cone(C).
$$
\end{definition}

\begin{definition}
Let $v_1,\dots,v_d$ be the other vertices of $C$ and let 
$$
\La=\bigoplus_{j=1}^d\Z e_j.
$$
where $e_j=r_jv_j$ and $r_j>0$ for each $j=1,\dots,d$ is the largest rational  number such that
$
\La_0\subset\La,
$
where $\La_0$ is the lattice of vertices of type zero.
\end{definition}

Let 
$$
\N^d(\La_0)
$$ 
denote the set of all 
$k=(k_1,\dots,k_d)\in \N^d$ such that $\sum_{j=1}^d k_je_j$ lies in $ \La_0$.
Analogously define
$$
\N_0^d(\La_0).
$$
For given $k\in\N_0^d$ the element $\sum_{j=1}^dk_je_j$ is contained in a unique chamber $C(k)\subset \Cone(C)$ such that $\Cone(C(k))\subset\Cone(C)$ as in the picture.
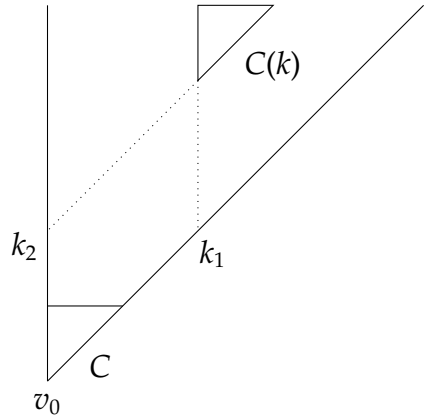
\begin{center}
\begin{tikzpicture}
\draw(0,0)--(5,5);
\draw(0,0)--(0,5);
\draw(0,1)--(1,1);
\draw(2,4)--(2,5)--(3,5)--(2,4);
\draw(0,-.3)node{$v_0$};
\draw(.7,.2)node{$C$};
\draw(3,4.2)node{$C(k)$};
\draw[dotted](2,4)--(2,2);
\draw[dotted](2,4)--(0,2);
\draw(2.2,1.7)node{$k_1$};
\draw(-.3,1.8)node{$k_2$};
\end{tikzpicture}
\end{center}
We say that the chamber $C(k)$ is \e{in relative position $k$} to $C$ and we write this as
$$
C\leadsto_k C(k).
$$

\section{Discrete groups}

The automorphism group $G=\Aut(X)$ of the building carries a natural topology, the compact-open topology. 
It is a locally-compact group which is totally disconnected.
A subgroup $\Ga\subset G$ is discrete if and only if for each chamber $C$ the stabilizer group
$$
\Ga_C=\big\{ \ga\in \Ga: \ga C=C\big\}
$$
is finite.
A discrete group $\Ga$ is a \e{lattice} in $G$, i.e., there exists a finite $G$-invariant Radon measure on $G/\Ga$ if and only if
$$
\sum_{C\in\CC}\frac1{|\Ga_C|}<\infty.
$$

We fix a discrete group $\Ga\subset\Aut(X)$.
The subgroup $\Ga^\mathrm{lab}$ of all $\ga\in\Ga$ which preserve a given labelling, is normal and of finite index in $\Ga$.
Replacing $\Ga$ by $\Ga^\mathrm{lab}$ we will henceforth \e{assume} that $\Ga$ preserves labellings.

\begin{definition}
Fix a discrete, label preserving subgroup $\Ga\subset G$.
A \e{$\Ga$-weight} is a map
$$
w:\CC\times\CC\to [0,\infty),
$$
such that
\begin{itemize}
\item $w(\ga C,D)=w(C,\ga D)=w(C,D)$ for all $C,D\in\CC$ and all $\ga\in\Ga$,
\item $\ds w(C,D)\ne 0\quad\Rightarrow\quad C\leadsto \ga D$ for some $\ga\in\Ga$,
\item $\ds\sum_{C,D\in\Ga\bs\CC}w(C,D)<\infty$,
\item $w(C,D)w(D,E)=w(C,E)$ holds for all $C,D,E\in\CC$ which satisfy $C\leadsto D\leadsto E$.
\end{itemize}
\end{definition}

\begin{definition}
Let $w$ be a $\Ga$-weight.
For $k\in\N_0^d(\La_0)$ we define an operator
$$
T_k:\ell^\infty(\CC)\to \ell^\infty(\CC)
$$
by
$$
T_k(C)=\sum_{C'} w(C,C')C',
$$
where the sum runs over all chambers $C'$ in relative position $k$ to $C$.
Note that for given $k\in\N_0^d$ in each apartment $\a$ containing $C$ there is at most one $C'$ in position $k$, but the same $C'$ can lie in infinitely many apartments containing $C$.
As we assume the building to be locally finite, the sum defining $T_k$ is actually finite.
\end{definition}

\begin{lemma}
For $k,l\in\N_0^d(\La_0)$ we have
$$
T_kT_l=T_{k+l}.
$$
In particular, the operators $T_k$ and $T_l$ commute.
\end{lemma}

\begin{proof}
On the one hand, the chamber $T_k(T_l(C))$ is in relative position $k+l$ and on the other, for any chamber $D$ in relative position $k+l$ to $C$ there exist uniquely determined chambers $C_k$ and $C_l$ in relative positions $k$ and $l$ such that each apartment containing $C$ and $D$ contains $C_1$ and $C_2$. This and the transitivity of $w$ proves the claim.
\end{proof}

\begin{definition}
A \e{quasicharacter} on $\N^d(\La_0)$ is a map
$\chi: \N^d(\La_0)\to\C^\times$ with $\chi(k+l)=\chi(k)\chi(l)$. 
For a given quasicharacter $\chi$ let
$$
\ell^2(\CC)_\chi
$$
be the generalized eigenspace, i.e, the set of all $v\in\ell^2(\CC)$ such that for every $k\in\N^d(
\La_0)$ one has
$$
(T_k-\chi(k))^mv=0
$$
for some $m\in\N$.
For every non-zero $\chi$ the space $\ell^2(\CC)_\chi$ is finite-dimensional.
Let $m(\chi)\in\N_0$ denote its dimension.
We then get
$$
\tr(T_k)=\sum_{\chi\in\CQ} m(\chi)\chi(k),
$$
where the sum runs over the set $\CQ$ of all  quasi-characters $\chi$.
\end{definition}

\begin{proposition}
For $k\in\N^d(\La_0)$ let $a_k\in\C$ be bounded, i.e., there exists $M>0$ such that $|a_k|\le M$ holds for all $k$.
Then the operator
$$
T=\sum_{0\ne k\in\N^k(\La_0)}a_kT_k
$$
is well defined and maps $\ell^2(\Ga\bs\CC)=\ell^2(\CC)^\Ga$ to itself.
On this Hilbert space, $T$ is a trace class operator.
\end{proposition}

\begin{proof}
We compute
\begin{align*}
\sum_{C\mod \Ga}\norm{TC}
&\le\sum_{\substack{C,D\mod\Ga\\C\leadsto_kD}}M w(C,D)<\infty.
\end{align*}
This implies that $T$ is trace class.
\end{proof}

\begin{definition}
Let $\CT$ denote the unital subring of $\End(\ell^\infty(\CC))$ generated by the translation operators $T_k$ with $k\in\N^d(\La_0)$.
This is a commutative integral domain.
Let $\CK$ denote its quotient field.

For indeterminates $u_1,\dots, u_d$ we define the formal power series
$$
T(u)=\sum_{k\in\N^d(\La_0)}u^kT_k\in\CT[[u_1,\dots,u_d]],
$$
where $u^k=u_1^{k_1}\cdots u_d^{k_d}$.
Note that the summation only runs over the set $\N^d(\La_0)$ of all $k\in\N^d$ such that $\sum_{j=1}^dk_je_j\in\La_0$.
\end{definition}

\begin{theorem}\label{thm2.4}
$T(u)$ is a rational function in $u$.
More precisely, there exists a finite set $E\subset \La_0^+$ and $k(e)\in\N_0^d$ for every $e\in E$ as well as $k(1),\dots,k(d)\in \N^d_0\sm\{0\}$ such that
$$
T(u)=\frac{\sum_{e\in E}u^{k(e)}T_{k(e)}}
{\(1-T_{k(1)}u^{k(1)}\)\cdots \(1-T_{k(d)}u^{k(d)}\)}.
$$
\end{theorem}

\begin{proof}
The case of trivial weight is Theorem 3.1.4 of \cite{bldglat}.
The proof given there extends without problems to the case of general weight.
\end{proof}

\begin{definition}
We say that $u\in\C^d$ is \e{singular} if there exists $1\le j\le d$ such that $u^{-k(j)}$ is an eigenvalue of $T_{k(j)}$.
Otherwise, $u$ is \e{regular}.
The singular set is a countable union of complex submanifolds of codimension 1 in $\C^d$, so the regular set is connected, open and dense. 
\end{definition}

\begin{proposition}
The family $u\mapsto T(u)$ is a meromorphic family of trace class operators on $\C^d$.
It is holomorphic on the regular set.
The map
$$
Z(u)=\tr T(u)
$$
is meromorphic on $\C^d$ and holomorphic on the regular set.
We have
\begin{align*}
Z(u)=\sum_{\chi\in\CQ}m(\chi)
\frac{\sum_{e\in E}u^{k(e)}\chi(k(e))}
{\(1-u^{k(1)}\chi(k(1))\)\cdots \(1-u^{k(d)}\chi(k(d))\)}
\end{align*}
\end{proposition}

\begin{proof}
This is clear from the theorem and the fact that all $T_k$ are trace class.
\end{proof}

\section{The prime geodesic theorem for a weighted building quotient}

Let $\D^\times=\{z\in\C: 0<|z|<1\}$.

\begin{theorem} [Prime Geodesic Theorem]
For $k\in\N^d(\La_0)$ let 
$$
N(k)=\sum_{\substack{C\mod\Ga\\ \exists_{\ga\in\Ga}C\leadsto\ga C}}w(C,\ga C).
$$
There are $z_1,\dots,z_N\in\C^d$ such that
$$
N(k)\ \sim\ \sum_{j=1}^N z_j^k,
$$
as $k_j\to\infty$ independently.
Moreover, there exists a sequence $z_j\in (\C^\times)^d$ with $\lim_jz_j=0$ such that for every $k\in\N^d(\La_0)$ we have, with absolute convergence of the sum,
$$
N(k)=\sum_{j=1}^\infty z_j^k.
$$
\end{theorem}

\begin{proof}
We show the last asertion first.
By definition we get $N(k)=\tr(T_k)$ and so
\begin{align*}
N(k)&=\sum_{\chi\in\CQ}m(\chi)\chi(k).
\end{align*}
As $\chi\in\CQ$ is a quasi-character, there exists $z_\chi\in\(\C^\times\)^d$ with
$$
\chi(k)=z_\chi^k.
$$
Now let $(z_j)$ be the sequence that runs through all $z_\chi$ with $m(\chi)\ne 0$, where each $z_\chi$ is repeated with multiplicity $m(\chi)$.
The claim follows.
The first claim follows from the fact that $z_j\to 0$ in $\C^d$.
\end{proof}

Writing $z_j=(z_{j,1},\dots,z_{j,d})\in(\C^\times)^d$ we write the statement of the theorem as
$$
N(k_1,k_2,\dots,k_d)=\sum_{j=1}^\infty z_{j,1}^{k_1}\cdots z_{j,d}^{k_d}.
$$

{\small Mathematisches Institut\\
Auf der Morgenstelle 10\\
72076 T\"ubingen\\
Germany\\
\tt deitmar@uni-tuebingen.de}

\today

\end{document}